\documentclass[12pt]{amsart}

\usepackage[all]{xy}
\usepackage{graphicx}
\usepackage{amssymb}
\usepackage{mathrsfs}
\usepackage{multirow}
\usepackage{float}
\usepackage{tikz-cd}
\usepackage{adjustbox}
\usepackage{amsthm}
\usepackage{accents}
\usepackage{mathtools}
\usepackage{mathptmx,enumitem,cite,array}
\usepackage[new]{old-arrows}
\usepackage{tikz}
\usetikzlibrary{matrix,arrows}
\usepackage{amsmath}
\usepackage{newtxmath}
\usepackage[T1]{fontenc}

\usepackage[numbers,sort&compress]{natbib}
\usepackage[bookmarksnumbered, bookmarksopen,
colorlinks,citecolor=blue,linkcolor=blue,backref]{hyperref}

\newcounter{RomanNumber}
\newcommand{\MyRoman}[1]{\setcounter{RomanNumber}{#1}\Roman{RomanNumber}}

\newtheorem{theorem}{Theorem}%[section]
\newtheorem{lemma}[theorem]{Lemma}

\theoremstyle{definition}

\newtheorem{proposition}[theorem]{Proposition}

\theoremstyle{remark}
\newtheorem{remark}[theorem]{Remark}

\theoremstyle{notation}
\newtheorem{notation}[theorem]{Notation}

\newcommand{\be}{\begin{equation}}
\newcommand{\ee}{\end{equation}}

\usepackage[utf8]{inputenc}
\usepackage{chngcntr}
\usepackage{apptools}
\AtAppendix{\counterwithin{theorem}{section}}
\usepackage[toc,page]{appendix}

\newcommand{\Conf}{\ensuremath{{\rm Conf}}}

\newcommand{\qqed}{\hfill\Box}

\begin{document}


\title
[Configuration of finite graphs]{Orders of the canonical vector bundles over configuration spaces of finite graphs}

\author{Frederick R. Cohen}
\address{Department of Mathematics, University of Rochester, Rochester, NY 14625, USA}
\email{cohf@math.rochester.edu}
\thanks{}

%  author two information
 \author{Ruizhi Huang}
\address{Institute of Mathematics and Systems Sciences, Chinese Academy of Sciences, Beijing 100190, China}
\email{huangrz@amss.ac.cn}
\urladdr{https://sites.google.com/site/hrzsea}
\thanks{}

%    \subjclass is required.
%\subjclass[2010]{Primary 58J26, Secondary 57S20, 53C27, 11F55, 19K35}
%\keywords{Elliptic genera, Witten genus, proper actions, loop Dirac induction, modularity, group $C^*$-algebras, representation ring, operator K-theory}
\date{}
%\thanks{}

\maketitle

{\centering\footnotesize {\it To the memory of Professor Wen-Ts\"{u}n Wu (1919-2017)
.}\par}

\begin{abstract}
We prove that the order of the canonical vector bundle over the configuration space is $2$ for a general planar graph, and is $4$ for a nonplanar graph.
\end{abstract}

%\tableofcontents
%--------------------------------------------------------------------------------------------------------------------------------------------------------------------------------------------------------%
%%%%%%%%%%%%%%%%%%%%%%%%%%%%%%%%%%%%%%%%%%%%%%%%%%%%%%%%%%%%%%%%%%%%%%%
\section*{Introduction}
Let $\xi$ be a vector bundle. If there exists a positive integer $n$ such that the $n$-fold Whitney sum $\xi^{\oplus n}$ is trivial, then we say that $\xi$ has finite order. In this case, the smallest such $n$ is called the {\it order} of $\xi$, denoted by $o(\xi)$. Meanwhile, if we are only interested in stable bundles and stable equivalences, there is a parallel notion of {\it stable order} of $\xi$, denoted by $s(\xi)$. It is obvious that
\[
s(\xi)~|~o(\xi).
\]

Let $\Conf(X, n)$ denote the the space of configurations of $n$ distinct points lying in a topological space $X$, that is,
\[
\Conf(X, n)=\{(x_1,x_2,\ldots, x_n)\in X\times \cdots \times X~|~ x_i\neq x_j ~{\rm for}~ i\neq j\}.
\]
If $X$ has at least $n$ distinct points, then $\Conf(X, n)$ is non-empty.
The symmetric group $\Sigma_n$ on $n$-letters acts freely on $\Conf(X,n)$ from the left by
\[
\sigma (x_1,x_2,\ldots, x_n)=(x_{\sigma(1)},x_{\sigma(2)},\ldots,x_{\sigma(n)}),~~~ \ \ \sigma\in \Sigma_n,
\]
which induces the canonical covering
\begin{equation}\label{coverFeq}
\Sigma_n\rightarrow \Conf(X,n)\rightarrow \Conf(X,n)/\Sigma_n.
\end{equation}
Since $\Sigma_n$ acts canonically on the real Euclidean space $\mathbb{R}^n$ by permuting the coordinates from the right, there is the associated vector bundle
\begin{equation}\label{canbundleeq}
\xi_{X, n}: \mathbb{R}^n\rightarrow \Conf(X,n)\times_{\Sigma_n} \mathbb{R}^n\rightarrow \Conf(X,n)/\Sigma_n.
\end{equation}

It is an enduring interest to determine the order and stable order of $\xi_{X, n}$ for various $X$. The order and stable order for $X=\mathbb{R}^n$ have been extensively studied. Cohen-Mahowald-Milgram \cite{CMM78} showed that $o(\xi_{\mathbb{R}^2,n})=2$. For the higher dimensional Euclidean spaces, there are studies by Yang \cite{Yang81} and Cohen-Cohen-Kuhn-Neisendorfer \cite{CCKN83}. Beyond the Euclidean case, Cohen-Cohen-Mann-Milgram \cite{CCMM89} showed that the order for oriented surface of genus greater or equal to one is $4$. Ren \cite{Ren17} studied the order for real projective spaces and their Cartesian products with a Euclidean space, and further he \cite{Ren18} investigated the order and stable order for simply connected spheres and their disjoint unions.

In this paper, we study the order of $\xi_{X,n}$ when $X$ is a finite (connected or non-connected) graph.
The configuration space of particles on finite graph is interesting in both mathematics and physics. For instance, Abrams in his thesis \cite{Abrams00} studied discrete model of configuration space on graph. The topology of configuration space on graph was studied by Farley-Sabalka \cite{FS05}, Barnett-Farber \cite{BF09}, Farber-Hanbury \cite{FH10}, etc, while the physical aspect on quantum statistics on graphs was investigated by Harrsion-Keating-Robbins-Sawicki \cite{HKRS14}, and Maci\k{a}\.{z}ek \cite{Maciazek}. 
Our main theorem is as follows.
\begin{theorem}\label{main0727}
Let $\Gamma$ be a finite graph. Then for any $n\geq 2$
\begin{itemize}
\item[(1).] if $\Gamma$ is homeomorphic to a point, or a closed interval, or a disjoint union of finitely many of them; or if $\Gamma$ is homeomorphic to a circle with $n$ odd, then 
\[
s(\xi_{\Gamma,n})=o(\xi_{\Gamma,n})=1;
\] 
\item[(2).] if $\Gamma$ is planar but does belong to case (1)
\[
s(\xi_{\Gamma,n})=o(\xi_{\Gamma,n})=2;
\]
\item[(3).] if $\Gamma$ is nonplanar
\[
s(\xi_{\Gamma,2})=o(\xi_{\Gamma,2})=4.
\]
\end{itemize}
\end{theorem}
\begin{proof}
It is well known that $\Conf(\mathbb{R}^1,n)$ is equivariantly homotopy equivalent to $\Sigma_n$, and the configuration space of a disjoint union is a disjoint union of products of configuration spaces of its components.  
Then the unordered configuration space $\Conf(\Gamma, n)/\Sigma_n$ in case (1) is homotopy equivalent to a disjoint union of points except the circle case. Hence the bundle $\xi_{\Gamma, n}$ is trivial and the order is $1$.
The order and stable order for $\Gamma\cong S^1$ are computed in Proposition \ref{s1orderpro}, while for the general planar graph they are determined in Proposition \ref{planarorderpro}. The orders for nonplanar graph are determined first for two kinds of special graphs homeomorphic to $K_5$ or $K_{3,3}$ in Proposition \ref{Kurnorderprop}, and then for the general cases in Proposition \ref{nonplanarorderprop}.
\end{proof}
\begin{notation}
Since the graphs in case 1 except ones homeomorphic to circle are trivial to our problem, we would like to exclude them in the computations. Hence, we let $\mathcal{L}$ be the set of graphs homeomorphic to a point, or a closed interval, or a disjoint union of finitely many of them, and usually suppose that $\Gamma \not\in \mathcal{L}$.
\end{notation}

The paper is organized as follows. In Section \ref{sec: model} we review the discrete model of Abrams for the configuration space of particles on graph. In Section \ref{sec: planar} and Section \ref{sec: nonplanar} we compute the orders and the stable orders of the canonical bundles for planar and nonplanar graphs respectively. Section \ref{sec: stable} is devoted to the application of Theorem \ref{main0727} on the stable homotopy types of generalized divided powers.

$\, $

\noindent{\em Acknowledgements.}
Ruizhi Huang was supported in part by National Natural Science Foundation of China (Grant no. 11801544), and ``Chen Jingrun'' Future Star Program of AMSS.

%--------------------------------------------------------------------------------------------------------------------------------------------------------------------------------------------------------%
%%%%%%%%%%%%%%%%%%%%%
\numberwithin{equation}{section}
\numberwithin{theorem}{section}

\section{}\label{sec: model}
Let us first recall some useful results about configuration space of particles on a finite graph, based on the thesis of Abrams \cite{Abrams00} and of Maci\k{a}\.{z}ek \cite{Maciazek}, and the papers \cite{BF09, HKRS14}.

Let $\Gamma$ be a finite graph, or equivalently a finite $1$-dimensional $CW$-complex. For any point $x\in \Gamma$, as in Section $1$ of \cite{BF09} we define the {\it support} of $x$ by
\[
 {\rm supp}\{x\}=\left\{\begin{array}{cc}
    x & \ \  {\rm if}~ x~{\rm is}~{\rm a}~{\rm vertex},   \\
    e & \ \ \  \ ~{\rm if}~ x\in \mathring{e}, ~{\rm an}~{\rm edge}.    \\
  \end{array}\right.
\]
For each $n\geq 2$, the {\it Abrams discrete model} of $\Conf(\Gamma, n)$ is defined to be 
\[
A(\Gamma, n)=\{(x_1,\ldots, x_n)\in \Gamma\times\cdots \times \Gamma~|~ {\rm supp}\{x_i\} \cap {\rm supp}\{x_j\} = \emptyset, ~ {\rm for}~{\rm all}~i\neq j\}.
\]
It is obvious that
\[
A(\Gamma, n)\subseteq \Conf(\Gamma, n),
\]
and the canonical permutation on $\Gamma\times\cdots \times\Gamma$ induces a free action on $A(\Gamma, n)$. Abrams proved the following important theorem in his thesis.
\begin{theorem}[Theorem $2.1$ of \cite{Abrams00}]\label{retrlemma}
Let $\Gamma$ be a finite graph with at least $n$ vertices. If $\Gamma$ satisfies that 
\begin{itemize}
\item[1.] each path between distinct vertices of degree not equal to $2$ passes through at least $n-1$ edges,
\item[2.] and each nontrivial loop passes through at least $n+1$ edges,
\end{itemize}
then the $\Sigma_n$-equivariant inclusion $A(\Gamma, n)\hookrightarrow \Conf(\Gamma, n)$ is a homotopy equivalence. In particular,
\[
\hspace{3.80cm}
A(\Gamma, n)/\Sigma_n\simeq \Conf(\Gamma, n)/\Sigma_n.  \hspace{3.80cm}\Box
\]  
\end{theorem}
Following \cite{Abrams00, FS05, HKRS14} we call a graph with properties 1 and 2 {\it sufficiently subdivided}.
When $n=2$, $A(\Gamma, 2)$ coincides with the ``simplicial deleted product'' of Shapiro \cite{Shapiro57}. In this case, Theorem \ref{retrlemma} can be strengthened.
\begin{lemma}[Theorem $2.4$ of \cite{Abrams00}]\label{eqretrlemma}
Let $\Gamma$ be a simple graph, i.e., a finite simplicial complex of dimension one. 
$A(\Gamma, 2)$ is a $\mathbb{Z}/2$-equivariant strong deformation retraction of $\Conf(\Gamma, 2)$. ~$\qqed$ 
\end{lemma}
This result was originally proved by Shapiro \cite{Shapiro57} and W-T Wu \cite{Wu59, Wu65}. However, as pointed out by Barnett-Farber \cite{BF09} the proof of Lemma $2.1$ of \cite{Shapiro57} is incorrect. Nevertheless, Abrams gave a clear proof in his thesis \cite{Abrams00}.

%--------------------------------------------------------------------------------------------------------------------------------------------------------------------------------------------------------%
\section{}\label{sec: planar}
Recall that a graph $\Gamma$ is {\it planar} if and only if it can be embedded into the real plane $\mathbb{R}^2$; otherwise $\Gamma$ is {\it nonplanar}. Moreover, if $\Gamma$ is nonplanar, it is easy to see that it can be embedded into orientable surface of higher genus (for instance, see page 53 of \cite{White01}). Let us consider planar graphs in this section.  We start with a general observation which has been used for example in \cite{CCMM89, Ren18} for surfaces and spheres. 
\begin{lemma}\label{genoblemma}
Let $n$ and $m$ be two positive integers such that $n\leq m$. For two complexes $X$ and $Y$, suppose there exists a $\Sigma_n$-equivariant map
\[
\Conf(X,n)\stackrel{}{\longrightarrow} \Conf(Y, m),
\]
where $\Sigma_n$ acts on $\Conf(Y, m)$ through a group monomorphism $\Sigma_n \subseteq \Sigma_m$. Then
\[
s(\xi_{X, n})\leq s(\xi_{Y, m}).
\]
Moreover, if further $m=n$ then
\[
o(\xi_{X, n})\leq o(\xi_{Y, m}).
\]
\end{lemma}
\begin{proof}
Denote by $f$ the equivariant map in the assumption. $f$ induces a map 
\[
\tilde{f}: \Conf(X,n)/\Sigma_n\rightarrow \Conf(Y, m)/\Sigma_m.
\]
Let $\epsilon^{i}$ be the trivial bundle of rank $i$.
Then $\xi_{Y, m}$ is pulled back to $\xi_{X,n}\oplus \epsilon^{m-n}$ along $\tilde{f}$. By definition $\xi_{Y, m}^{\oplus s(\xi_{Y, m})}$ is stably trivial, which implies that $\xi_{X, n}^{\oplus s(\xi_{Y, m})}$ is stably trivial. Hence $s(\xi_{X, n})\leq s(\xi_{Y, m})$. When $m=n$, $o(\xi_{X, n})\leq o(\xi_{Y, m})$ by the similar argument and the lemma is proved.
\end{proof}

\begin{lemma}\label{planelemma}
Let $\Gamma$ be a planar finite graph. Then
\[
s(\xi_{\Gamma, n})=o(\xi_{\Gamma, n})=1, ~{\rm or}~2.
\]
\end{lemma}
\begin{proof}
Let $\Gamma \hookrightarrow \mathbb{R}^2$ be an embedding. It induces a $\Sigma_n$-equivariant embedding
\[
\Conf(\Gamma, n)\hookrightarrow \Conf(\mathbb{R}^2, n).
\]
It was proved in Theorem $1.2$ of \cite{CMM78} that $o(\xi_{\mathbb{R}^2, n})=2$. Hence by Lemma \ref{genoblemma}
$o(\xi_{\Gamma, n})$ can only be $1$ or $2$, and the lemma follows.
\end{proof}

The circle is special among all graphs and we may treat it first.

\begin{proposition}\label{s1orderpro}
Let $\Gamma$ be a finite graph homeomorphic to $S^1$. Then
\[
s(\xi_{\Gamma, n})=o(\xi_{\Gamma, n})=\left\{\begin{array}{ll}
1  & ~ n~{\rm is}~{\rm odd}\\
2 &  ~n~{\rm is}~{\rm even.}
\end{array}
\right.
\]
\end{proposition}
\begin{proof}
It is well known that $\Conf(S^1, n)\simeq \mathop{\coprod}\limits_{(n-1)!}S^1$, 
$\Conf(S^1, n)/\Sigma_n\simeq S^1$, and there is the map of coverings
\begin{gather*}
\begin{aligned}
\xymatrix{
\mathbb{Z}/n \ar[r] \ar[d]^{i} & S^1 \ar[r] \ar[d]^{j} & S^1 \ar@{=}[d]\\
\Sigma_n       \ar[r]  & \mathop{\coprod}\limits_{(n-1)!}S^1 \ar[r] & S^1, 
}
\end{aligned}
\label{redgps1diag}
\end{gather*}
where the first row is the canonical $n$-fold covering of $S^1$, the second row is homotopic to the covering of configuration space (\ref{coverFeq}) of $X=S^1$, $i$ is the injection of the subgroup consisting of cycles of length $n$, and $j$ is the inclusion of any component of $\mathop{\coprod}\limits_{(n-1)!}S^1$. In particular, the classifying map $f$ of $\xi_{S^1, n}$ can be factored as 
\begin{equation}\label{factorfeq}
f: S^1\stackrel{\tilde{f}}{\longrightarrow} B\mathbb{Z}/n\stackrel{Bi}{\longrightarrow} B\Sigma_n 
\stackrel{B\rho}{\longrightarrow} BO(n),
\end{equation}
where $\tilde{f}$ represents a generator of $\pi_1(B\mathbb{Z}/n)\cong \mathbb{Z}/n$, and $\rho$ is the canonical representation into the orthogonal group $O(n)$.
Equivalently, this means that the structure group of $\xi_{S^1, n}$ can be lifted to $\mathbb{Z}/n$.
In particular, $\omega_1(\xi_{S^1, n})=0$ when $n$ is odd, and $\xi_{S^1, n}$ is trivial. This proves the proposition when $n$ is odd.

On the other hand, there is the commutative diagram of group homomorphisms
\begin{gather*}
\begin{aligned}
\xymatrix{
\Sigma_n \ar[rr]^{\rho} \ar[dr]_{{\rm sgn}~~~} && O(n) \ar[dl]^{q~~~}\\
&\{-1, +1\}\cong \mathbb{Z}/2,
}
\end{aligned}
\label{snondiag}
\end{gather*}
where ${\rm sgn}$ is defined by the sign of permutation, and $q$ is the quotient of $O(n)$ by the subgroup $SO(n)$. Additionally, $(B{\rm sgn})^\ast: H^1(B\mathbb{Z}/2;\mathbb{Z}/2)\rightarrow H^1(B\Sigma_n;\mathbb{Z}/2)$ is an isomorphism, and $Bq$ represents the universal first Stiefel-Whitney class $\omega_1\in H^1(BO(n);\mathbb{Z}/2)$. It follows that $(B\rho)^\ast(\omega_1)\neq 0$, and further when $n$ is even $(Bi\circ B\rho)^\ast(\omega_1)\neq 0$ as ${\rm sgn}\circ i$ is surjective.
Hence, from (\ref{factorfeq}) we see that $\omega_1(\xi_{S^1, n})\neq 0$ when $n$ is even. Then $\xi_{S^1, 2}$ is stably nontrivial and $s(\xi_{\Gamma, n})=o(\xi_{\Gamma, n})=2$ by Lemma \ref{planelemma}. This proves the proposition when $n$ is even.
\end{proof}
\begin{remark}\label{mobiusremark}
When $n=2$, the total space of $\xi_{S^1, 2}$ is the M\"{o}bius strip, and the bundle is the projection to its equatorial circle.
\end{remark}

\begin{figure}[H]
\centering
\includegraphics[width=3.2in]{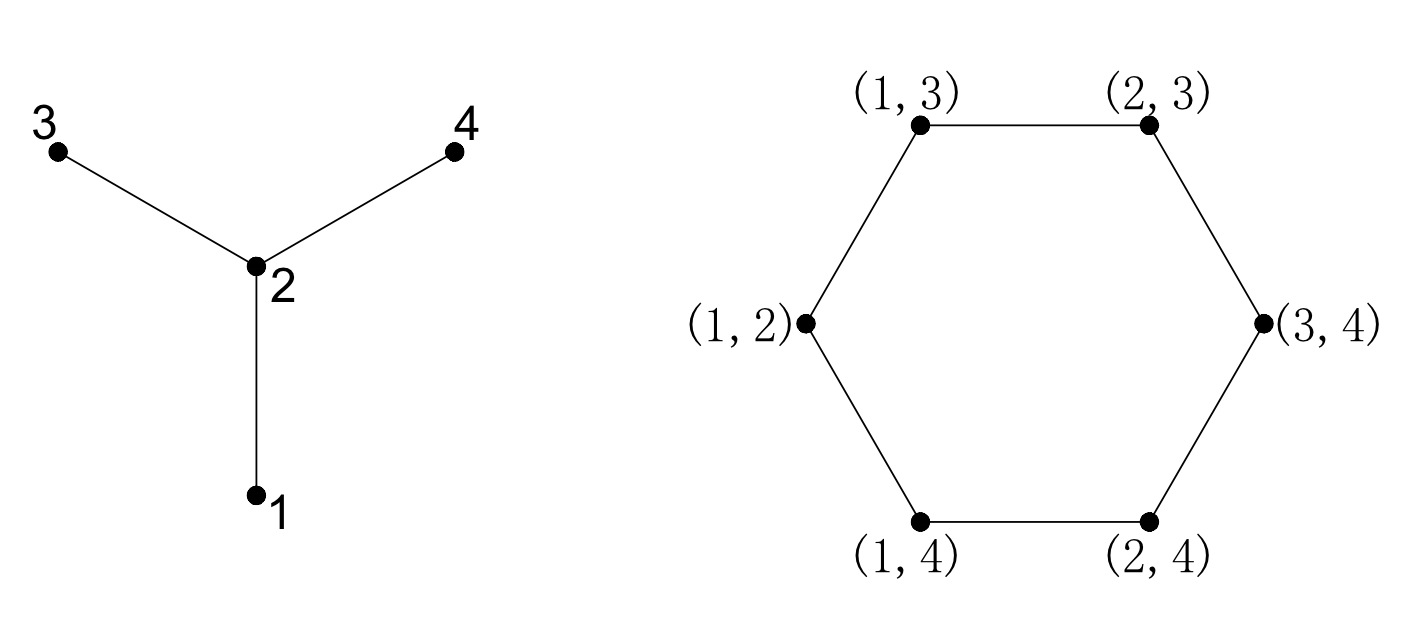}
\caption{$Y$-graph $K$ and $A(K,2)/\Sigma_2$}\label{ygraphfigure}
\end{figure}
Recall in Introduction we denote $\mathcal{L}$ to be the set of graphs homeomorphic to a point, or a closed interval, or a disjoint union of finitely many of them.
\begin{proposition}\label{planarorderpro}
Let $\Gamma$ be a planar finite graph such that $\Gamma \not\in \mathcal{L}$ and $\Gamma \not\cong S^1$. Then
\[
s(\xi_{\Gamma, n})=o(\xi_{\Gamma, n})=2,
\]
for any $n\geq 2$.
\end{proposition}
\begin{proof}
Let us firstly consider the case when $\Gamma$ is connected. By assumption $\Gamma$ contains a vertex of degree at least $3$, in other words, it contains a $Y$-subgraph $K$ (Figure \ref{ygraphfigure}). By Example 2.2 of \cite{Abrams00}, $\Conf(K,2)/\Sigma_2\simeq S^1$. Then as in Remark \ref{mobiusremark}, $\xi_{K,2}$ is the M\"{o}bius strip, and $s(\xi_{K, 2})=o(\xi_{K, 2})=2$ by Lemma \ref{planelemma}.

For general $n$, we may first choose any proper self-embedding $i: K\hookrightarrow K$ of $K$.
With this we can define a $\Sigma_2$-equivariant map
\begin{equation}\label{yembeq}
\tilde{i}: \Conf(K,2)\stackrel{}{\longrightarrow} \Conf(K,n)
\end{equation}
by sending $(x, y)$ to $(x,y, a_1, a_2,\ldots, a_{n-2})$, where the $(n-2)$ distinct points $a_1$, $a_2,\ldots, a_{n-2}\in K-i(K)$.
Then there is the composition of $\Sigma_2$-equivariant maps
\begin{equation}\label{confk2gneq}
k: \Conf(K,2)\stackrel{\tilde{i}}{\longrightarrow} \Conf(K,n)\stackrel{j}{\longrightarrow} \Conf(\Gamma, n),
\end{equation}
where $j$ is the canonical inclusion. 
By Lemma \ref{genoblemma} $s(\xi_{\Gamma, n})\geq s(\xi_{K, 2})=2$, 
and the proposition follows from Lemma \ref{planelemma} for the case when $\Gamma$ is connected.

When $\Gamma$ is not connected, there is a component of $\Gamma$ containing a subgraph $T$, which is either a $Y$-graph or homeomorphic to $S^1$ as $\Gamma \not\in\mathcal{L}$. If $T$ is a $Y$-graph, then the previous discussion implies the statement of the proposition. Now suppose $T$ is homeomorphic to $S^1$. Consider the canonical $\Sigma_{n}$-equivariant embedding $K(T,n)\hookrightarrow K(\Gamma, n)$ when $n$ is even, while consider the $\Sigma_{n-1}$-equivariant embedding $K(T,n-1)\hookrightarrow K(\Gamma, n)$ sending $(x_1,\ldots, x_{n-1})$ to $(x_1,\ldots, x_{n-1}, y)$ with $y$ lying in a component of $\Gamma$ different from that of $T$ when $n$ is odd. 
Then by Lemma \ref{s1orderpro} and Lemma \ref{planelemma} and the similar argument above, we see that in either case $s(\xi_{\Gamma, 2n})=o(\xi_{\Gamma, 2n})=2$. This completes the proof of the proposition.
\end{proof}
\begin{remark}\label{Ynonplaneremark}
From the proof of Proposition \ref{planarorderpro}, we also see that the stable order of $\xi_{\Gamma, n}$ can not be $1$ for any nonplanar finite graph $\Gamma$ since it contains a proper $Y$-subgraph.
\end{remark}

%--------------------------------------------------------------------------------------------------------------------------------------------------------------------------------------------------------%
\section{}\label{sec: nonplanar}
In this section, we determine the order and stable order of $\xi_{\Gamma, n}$ for nonplanar graph $\Gamma$.
\begin{lemma}\label{nonplanarorderlemma}
Let $\Gamma$ be a nonplanar finite graph. Then
\[
o(\xi_{\Gamma, n})=2,~{\rm or}~4,
\]
for any $n\geq 2$.
\end{lemma}
\begin{proof}
It was showed in \cite{CCMM89} that the order of $\xi_{M, n}$ is $4$ for any closed orientable Riemann surface $M$ of genus greater than or equal to one. Then by the fact that $\Gamma$ can be embedded into orientable surface of higher genus and Lemma \ref{genoblemma} , $o(\xi_{\Gamma, n})$ can only $1$, $2$ or $4$. However, $1$ is impossible by Remark \ref{Ynonplaneremark}.
\end{proof}

Two famous examples of non-planar graphs are the complete graph on five vertices $K_5$ and the complete bipartite graph $K_{3,3}$. A {\it Kuratowski graph} is a subdivision of $K_5$ or $K_{3,3}$.
Here, a {\it subdivision} of a graph $\Gamma$ is a graph resulting from the subdivision of edges of $\Gamma$ by introducing new vertices on them. The following criterion of nonplanar graph is classical.
\begin{theorem}[Kuratowski’s Theorem, 1930; \cite{Kur30}]\label{Kuthm}
A graph $\Gamma$ is nonplanar if and only if $\Gamma$ contains a Kuratowski subgraph. ~$\qqed$
\end{theorem}

\begin{lemma}[Abrams; Section $5.1$ in \cite{Abrams00}]\label{k533=surfacelemma}
\[
A(K_5, 2)\simeq  \sharp_{6} T^2, \ \  A(K_{3,3}, 2)\simeq  \sharp_{4} T^2,
\]
where $\sharp_{k}T^2$ is the orientable closed surface of genus $k$. ~$\qqed$
\end{lemma}
In the following, Proposition \ref{Kur2orderprop} is a special case of Proposition \ref{Kurnorderprop}, but it is proved in a different and simpler way.
\begin{proposition}\label{Kur2orderprop}
Let $\Gamma$ be a Kuratowski graph. Then 
\[
s(\xi_{\Gamma, 2})=o(\xi_{\Gamma, 2})=4.
\]
\end{proposition}
\begin{proof}
First since there is the covering map
\[
\mathbb{Z}/2\rightarrow \sharp_{2k} T^2\rightarrow \sharp_{2k+1} P^2,
\]
where $\sharp_{2k+1} P^2$ is the unorientable closed surface of genus $2k+1$,
we see from Lemma \ref{eqretrlemma} and Lemma \ref{k533=surfacelemma} that
\[
\begin{split}
& \Conf(K_5, 2)/\Sigma_2 \simeq A(K_5, 2)/\Sigma_2 \simeq \sharp_{7} P^2\\
 &\Conf(K_{3,3}, 2)/\Sigma_2 \simeq A(K_{3,3}, 2)/\Sigma_2 \simeq \sharp_{5} P^2.
\end{split}
\]
Further, notice that the determinant line bundle of $\xi_{\Gamma,2}$ is determined by the connecting epimorphism
\[
h: \pi_{1}(\Conf(\Gamma, 2)/\Sigma_2)\rightarrow \Sigma_2\cong \mathbb{Z}/2,
\]
which, as the orientation character, corresponds exactly to the first Stiefel-Whitney class $\omega_1(\sharp_{2k+1} P^2)$ with $k=3$, or $2$. Hence, in either case
\[
\omega_1^2(\xi_{\Gamma,2})=\omega_1^2(\sharp_{2k+1} P^2)\neq 0.
\]
Since $\xi_{\Gamma,2}$ is isomorphic to the direct sum of a line bundle and the trivial line bundle, it follows that
\[
\omega(\xi_{\Gamma,2}^{\oplus 2})=(1+\omega_1(\xi_{\Gamma,2}))^2=1+\omega_1^2(\sharp_{2k+1} P^2)\neq 1.
\]
Hence both the order $o(\xi_{\Gamma,2})$ and the stable order $s(\xi_{\Gamma,2})$ can not be $2$. The proposition then follows from Lemma \ref{nonplanarorderlemma}.
\end{proof}

\begin{proposition}\label{Kurnorderprop}
Let $\Gamma$ be a Kuratowski graph. Then 
\[
s(\xi_{\Gamma, n})=o(\xi_{\Gamma, n})=4,
\]
for any $n\geq 2$.
\end{proposition}
\begin{proof}
The case when $n=2$ was showed in Proposition \ref{Kur2orderprop}. For general $n$, 
\begin{equation}\label{H1uconfkurneq}
H_1(\Conf(\Gamma, n)/\Sigma_n;\mathbb{Z})\cong H_1(\Conf(\Gamma, 2)/\Sigma_2;\mathbb{Z})\cong \mathop{\oplus}\limits_{2k}\mathbb{Z}\oplus \mathbb{Z}/2,
\end{equation}
 where $k=3$, or $2$ according to $\Gamma\cong K_5$ or $K_{3,3}$ by Theorem 5 of \cite{HKRS14} and Lemma \ref{k533=surfacelemma}. Moreover, by the discussion before Theorem 5 of \cite{HKRS14}, both the $\mathbb{Z}/2$ summands in the homology are determined by a $Y$-subgraph $K$ of $\Gamma$. As in the proof of Proposition \ref{planarorderpro}, we can define a $\Sigma_2$-equivariant embedding $\tilde{i}: \Conf(K,2)\stackrel{}{\rightarrow} \Conf(K,n)$ (\ref{yembeq}) from any proper self-embedding $i: K\hookrightarrow K$ of $K$, and similar to (\ref{confk2gneq}) consider the composition of $\Sigma_2$-equivariant maps
\[
k: \Conf(K,2)\stackrel{\tilde{i}}{\longrightarrow} \Conf(K,n)\stackrel{j}{\longrightarrow} \Conf(\Gamma, n),
\]
where $j$ is the canonical embedding.
It follows that $k$ induces a map
\[
\tilde{k}: \Conf(K,2)/\Sigma_2\rightarrow \Conf(\Gamma, n)/\Sigma_n.
\]
Then there is the commutative diagram
\begin{gather*}
\begin{aligned}
\xymatrix{
\mathbb{Z}\cong H_1(\Conf(K,2)/\Sigma_2;\mathbb{Z}) \ar[r]^{\ \ \tilde{k}_\ast} \ar[d]^{\rho_2}  & H_1(\Conf(\Gamma, n)/\Sigma_n;\mathbb{Z}) \ar[d]^{\rho_2}  \ar[r]^>>>>>{p} &\mathbb{Z}/2 \ar@{=}[d]\\
\mathbb{Z}/2\cong H_1(\Conf(K,2)/\Sigma_2;\mathbb{Z}/2) \ar[r]^>>>>>{\tilde{k}_\ast} & H_1(\Conf(\Gamma, n)/\Sigma_n;\mathbb{Z}/2) \ar[r]^>>>>>{p} & \mathbb{Z}/2,
}
\end{aligned}
\label{rho2hom1diag}
\end{gather*}
where both $\rho_2$ are the mod-$2$ reductions, and both $p$ are the projections onto the $\mathbb{Z}/2$-summands determined by the $Y$-subgraph $K$. Notice that the composition of maps in the top row is the mod-$2$ reduction. It follows that the composition of maps in the bottom row is an isomorphism. Then since by Lemma \ref{s1orderpro} $\omega_1(\xi_{K,2})\neq 0$ is the generator of $H^1(\Conf(K,2)/\Sigma_2;\mathbb{Z}/2)$, $\omega_1(\xi_{\Gamma,n})\neq 0$.
Moreover, by (\ref{H1uconfkurneq}) the Bockstein homomorphism
\[
\beta=Sq^1: H^1(\Conf(\Gamma, n)/\Sigma_n;\mathbb{Z}/2)\cong \mathop{\oplus}\limits_{2k}\mathbb{Z}/2\oplus \mathbb{Z}/2\stackrel{}{\longrightarrow} H^2(\Conf(\Gamma, n)/\Sigma_n;\mathbb{Z}/2)
\]
is trivial on the $\mathop{\oplus}\limits_{2k}\mathbb{Z}/2$-summand and is nontrivial on the last $\mathbb{Z}/2$-summand corresponding to the $Y$-subgraph $K$.
Therefore $\omega_1(\xi_{\Gamma,n})$ is the generator of the last $\mathbb{Z}/2$ and satisfies $\omega_1^2(\xi_{\Gamma,n})=Sq^1(\omega_1(\xi_{\Gamma,n}))\neq 0$. 
Hence
\[
\omega(\xi_{\Gamma,n}^{\oplus 2})\equiv 1+\omega_1^2(\xi_{\Gamma,n})~{\rm mod}~{H^{\geq 3}(\Conf(\Gamma, n)/\Sigma_n;\mathbb{Z}/2)}
\]
is not $1$, which implies that $\xi_{\Gamma,n}^{\oplus 2}$ is not stably trivial, and the stable order $s(\xi_{\Gamma,2})$ can not be $2$. The proposition then follows from Lemma \ref{nonplanarorderlemma}.
\end{proof}

The general case can then be determined by either Proposition \ref{Kur2orderprop} or Proposition \ref{Kurnorderprop}.
\begin{proposition}\label{nonplanarorderprop}
Let $\Gamma$ be a nonplanar finite graph but not a Kuratowski graph. Then
\[
s(\xi_{\Gamma, n})=o(\xi_{\Gamma, n})=4,
\]
for any $n\geq 2$.
\end{proposition}
\begin{proof}
By assumption, $\Gamma$ contains a proper Kuratowski subgraph $K$ by Theorem \ref{Kuthm}. Choose $n-2$ distinct points $a_1$, $a_2,\ldots, a_{n-2}\in \Gamma-K$. Then there is a $\Sigma_2$-equivariant embedding 
\[
\widetilde{\Phi}: \Conf(K,2)\hookrightarrow \Conf(\Gamma, n)
\]
sending $(x, y)$ to $(x,y, a_1, a_2,\ldots, a_{n-2})$. Hence, by Lemma \ref{genoblemma} and Proposition \ref{Kur2orderprop}
\[
o(\xi_{\Gamma, n})\geq s(\xi_{\Gamma, n})\geq s(\xi_{K,2})=4.
\]
The proposition now follows from Lemma \ref{nonplanarorderlemma}.
\end{proof}
%--------------------------------------------------------------------------------------------------------------------------------------------------------------------------------------------------------%
\section{}\label{sec: stable}
Given two complexes $X$ and $Z$, we can consider the so-called {\it $n$-th generalized divided power} of $Z$ associated to $X$ for any $n\geq 2$ defined by
\[
D_n(X, Z):=\Conf(X, n)^{+}\wedge_{\Sigma_n} Z^{\wedge n},
\]
where $Z^{\wedge n}$ is the $n$-fold self-smashed product of $Z$ and inherits a $\mathbb{Z}_n$-action from the canonical permutation on $Z^{\times n}$. 
Let $\Sigma^i Y$ be the $i$-fold suspension of the complex $Y$.
The significance of the generalized divided powers is due to a general version of Snaith's stable splitting \cite{Snaith74}. Indeed, following \cite{CT78} define the {\it labelled configuration space of $X$ with labels in $Z$} by
\[
\Conf(X, Z):=\mathop{\coprod}\limits_{n} \Conf(X, n)\times_{\Sigma_n} Z^{\times n}/\sim,
\]
where the equivalence relation $\sim$ is generated by
\[
\{(x_1, \ldots, x_n, z_1,\ldots z_n) \sim (x_1, \ldots, x_{n-1}, z_1,\ldots, z_{n-1}), \ ~~~{\rm if}~z_n=\ast \}
\]
with $\ast$ the based point of $Z$. Then when $Z$ is path connected there is the stable decomposition \cite{CT78, CMT78}
\[
\Sigma^{\infty} \Conf(X, Z)\simeq \mathop{\bigvee}\limits_{n=1}^{\infty} \Sigma^{\infty} D_n(X, Z).
\]

The following lemma is due to an unpublished manuscript of Cohen and was reproved by Ren in \cite{Ren18}.
\begin{lemma}[Lemma 6.1 and Corollary 6.2 of \cite{Ren18}]\label{cohenrenlemma}
For any positive integer $t$ and $n\geq 2$, there is a homotopy equivalence
\[
\hspace{3.13cm}
\Sigma^{nto(\xi_{X, n})} D_n(X, Z)\stackrel{}{\longrightarrow} D_n(X, \Sigma^{to(\xi_{X, n})} Z).
 \hspace{3.13cm}\Box
\]
\end{lemma}
By Lemma \ref{cohenrenlemma} and Theorem \ref{main0727}, we immediately obtain the following proposition, which indicates that the stable homotopy types of $D_n(X, \Sigma^t Z)$ exhibit a natural periodic behavior as $t$ varies.

\begin{proposition}
Let $\Gamma$ be a finite graph such that $\Gamma\not\in \mathcal{L}$. Then for any complex $Z$,  positive integer $t$ and $n\geq 2$
\begin{itemize}
\item[(1).] if $\Gamma$ is homeomorphic to a circle with $n$ odd, then 
\[
\Sigma^{nt} D_n(\Gamma, Z)\stackrel{}{\longrightarrow} D_n(\Gamma, \Sigma^{t} Z);
\] 
\item[(2).] if $\Gamma$ is homeomorphic to a circle with $n$ even; or if $\Gamma$ is planar such that $\Gamma \not\cong S^1$ 
\[
\Sigma^{2nt} D_n(\Gamma, Z)\stackrel{}{\longrightarrow} D_n(\Gamma, \Sigma^{2t} Z);
\]
\item[(3).] if $\Gamma$ is nonplanar
\[
\Sigma^{4nt} D_n(\Gamma, Z)\stackrel{}{\longrightarrow} D_n(\Gamma, \Sigma^{4t} Z).
\]
\end{itemize}
\end{proposition}
%%%%%%%%%%%%%%%%%%%%%%%%%%%%%%%%%%%%%%%%%%%%%%%%%%%%%%%%%%%%%%%%%%%%%%%
%%%%%%%%%%%%%%%%%%%%%%%%%%%%%%%%%%%%%%%%%%%%%%%%%%%%%%%%%%%%%%%%%%%%%%%

\end{document}